\documentclass[12pt]{amsart}
\usepackage[utf8]{inputenc}
\usepackage{amsmath}
\usepackage{xcolor}
\usepackage{amssymb}
\newtheorem{theorem}{Theorem}
\newtheorem*{theorem*}{Theorem}

\newtheorem{corollary}{Corollary}[section]
\newtheorem{lemma}{Lemma}

\newtheorem*{acknowledgements*}{Acknowledgements}

\makeatletter
\def\blfootnote{\gdef\@thefnmark{}\@footnotetext}
\makeatother

\def\house#1{\setbox1=\hbox{$\,#1\,$}%
\dimen1=\ht1 \advance\dimen1 by 2pt \dimen2=\dp1 \advance\dimen2 by 2pt
\setbox1=\hbox{\vrule height\dimen1 depth\dimen2\box1\vrule}%
\setbox1=\vbox{\hrule\box1}%
\advance\dimen1 by .4pt \ht1=\dimen1
\advance\dimen2 by .4pt \dp1=\dimen2 \box1\relax}

\begin{document}
\title{Distribution of residues modulo $p$ using the Dirichlet's class number formula}
\author{Jaitra Chattopadhyay, Bidisha Roy, Subha Sarkar and R. Thangadurai}
   \address[Jaitra Chattopadhyay, Bidisha Roy, Subha Sarkar and R. Thangadurai]{Harish-Chandra Research Institute, HBNI\\
Chhatnag Road, Jhunsi\\
Allahabad 211019, India}
\email[Jaitra Chattopadhyay]{jaitrachattopadhyay@hri.res.in}
\email[Bidisha Roy]{bidisharoy@hri.res.in}
\email[Subha Sarkar]{subhasarkar@hri.res.in}
\email[R. Thangadurai]{thanga@hri.res.in}

\begin{abstract}
Let $p$ be  an odd prime number.  In this article, we study the number  of
quadratic  residues and non-residues modulo $p$ which are multiples of $2$ or $3$ 
or  $4$ and  lying in the interval $[1, p-1]$,  by applying  the Dirichlet's class
number formula for the imaginary quadratic field $\mathbb{Q}(\sqrt{-p})$. 
\end{abstract}
\maketitle

\section{Introduction}
Let $p$ be an odd prime number. A number $a \in \{1,\ldots,p-1\}$ is said to be a {\it quadratic residue} modulo $p$, if the congruence
\begin{equation*}\label{eq1}
x^2 \equiv a \pmod p
\end{equation*}
has a solution in $\mathbb{Z}$. Otherwise, $a$ is said to be a {\it quadratic non-residue} modulo $p$.  The study of distribution of  quadratic residues and quadratic non-residues  modulo $p$  has been considered with great interest in the literature (see for instance, 
\cite{br1}, \cite{car}, \cite{gic1}, \cite{tha1}, \cite{tha2}, \cite{hau1}, \cite{kohnen}, \cite{tha3}, \cite{thanga}, \cite{mos}, \cite{poll1}, \cite{sza1}, \cite{sza2}, \cite{vegh1}, \cite{vegh2}, \cite{vegh3}, \cite{vegh4} and \cite{wright}).

\smallskip

Since $\mathbb{Z}/p\mathbb{Z}$ is a field, the polynomial $X^{p-1} -1$ has precisely $p-1$ non-zero solutions over $\mathbb{Z}/p\mathbb{Z}$. As $p$ is an odd prime, we see that $X^{p-1}-1 = (X^{(p-1)/2}+ 1)(X^{(p-1)/2}-1)$ and  one can conclude that   there are exactly $\frac{p-1}{2}$ quadratic residues as well as non-residues modulo $p$ in the interval $[1, p-1]$. 

\smallskip

\noindent{\bf Question 1.}  {\it For an odd prime number $p$ and a given natural number $k$ with $1\leq k\leq p-1$, we let $S_k = \{a \in \{1, 2, \ldots, p-1\} \ : \ a\equiv 0\pmod{k}\}$ be the subset consisting of all natural numbers which are multiples of  $k$.  How many  quadratic residues (respectively, non-residues) lie inside $S_k$?}

\bigskip

In the literature, there are many papers addressed similar to Question 1 and to name a few, one may refer to \cite{joh1},   \cite{vas1} and \cite{mig1}.  First we shall fix some notations as follows. We denote by $Q(p, S_k)$ (respectively, $N(p, S_k)$) the number of quadratic residues (respectively, quadratic non-residues) modulo $p$ in the subset $S_k$ of the interval $[1, p-1]$.

\bigskip

The standard techniques in analytic number theory  answers the above question as 
\begin{equation}\label{equation0}
Q(p, S_k) =\frac{p-1}{2k} +O(\sqrt{p}\log p)
\end{equation}
and the same result is true for $N(p, S_k)$ for all $k$ (we shall be proving this fact in this article). However, it might happen that for some primes $p$, we may have 
$Q(p, S_k) > N(p, S_k)$ or $Q(p, S_k) < N(p, S_k)$.  Using the standard techniques, we could not answer this subtle question. In this article, we shall answer this using the Dirichlet's class number formula for the field $\mathbb{Q}(\sqrt{-p})$,  when $k = 2, 3$ or $4$.  More precisely, we prove the following theorems. 

\begin{theorem}\label{thm2}
Let $p$ be an odd prime. If $p\equiv 3\pmod{4}$, then for  any $\epsilon$ with $0 < \epsilon < \frac{1}{2}$, we have
$$
Q(p, S_2) - \frac{p-1}{4} \gg_\epsilon p^{\frac{1}{2}-\epsilon}.
  $$
When the prime $p\equiv 1\pmod{4}$, we have 
$$
Q(p, S_2) =  \frac{p-1}{4}.
$$
\end{theorem}

\begin{corollary}\label{thm2-cor1}
Let $p$ be an odd prime and let $\mathcal{O}$ be the set of all odd integers in $[1, p-1]$.
If $R = N(p, S_2)$ or $R = Q(p, \mathcal{O})$, then for any $\epsilon$ with $0<\epsilon < \frac{1}{2}$, we have 
$$\frac{p-1}{4} - R \gg_\epsilon p^{\frac{1}{2}-\epsilon},  \mbox{ if } p\equiv 3\pmod{4}.$$
When the prime $p\equiv 1\pmod{4}$, we have 
$$
R=  \frac{p-1}{4}.
$$
\end{corollary}

\begin{theorem}\label{thm3}
Let $p$ be an odd prime. If $p\equiv 1, 11\pmod{12}$, then for any $\epsilon$ with  $0 < \epsilon < \frac{1}{2}$, we have
$$
Q(p, S_3) - \frac{p-1}{6} \gg_\epsilon 
p^{\frac{1}{2}-\epsilon}.$$ 
\end{theorem}

When $p\equiv 5, 7\pmod{12}$, in this method, we do not get any finer information other than in  \eqref{equation0}.

\begin{corollary}\label{thm3-cor1}
Let $p$ be an odd prime. If $p\equiv 1, 11\pmod{12}$, then  for any $\epsilon$ with $0<\epsilon <\frac{1}{2}$, we have
$$ \frac{p-1}{6} - N(p, S_3) \gg_\epsilon p^{\frac{1}{2}-\epsilon}.
$$ 
\end{corollary}

\begin{theorem}\label{thm4}
Let $p$ be an odd prime. Then,  for $p\equiv 3\pmod{8}$, we have 
$$
Q(p, S_4) = \displaystyle \frac{1}{2}\left[ \frac{p-1}{4} \right].
$$
Also, for any $ 0< \epsilon <\frac{1}{2}$,  we have  
$$
Q(p, S_4)- \displaystyle \frac{p-1}{8} \gg_\epsilon p^{\frac{1}{2}-\epsilon}, \mbox{ if } p\equiv 1\pmod{4},
$$
and
$$
Q(p, S_4)- \displaystyle \frac{1}{2}\left[ \frac{p-1}{4} \right] \gg_\epsilon p^{\frac{1}{2}-\epsilon}; \mbox{ if } p\equiv 7\pmod{8}.
$$
\end{theorem}

\begin{corollary}\label{thm4-cor1}
Let $p$ be an odd prime. Then, for $p\equiv 3\pmod{8}$, we have 
$$
N(p, S_4) = \displaystyle \frac{1}{2}\left[ \frac{p-1}{4} \right].
$$
Also, for any $0<\epsilon<\frac{1}{2}$,  we have  
$$
\displaystyle \frac{p-1}{8} - N(p, S_4) \gg_\epsilon p^{\frac{1}{2}-\epsilon}; \mbox{ if } p\equiv 1\pmod{4},
$$
and
$$
\displaystyle \frac{1}{2}\left[ \frac{p-1}{4} \right]-N(p, S_4) \gg_\epsilon p^{\frac{1}{2}-\epsilon} ; \mbox{ if } p\equiv 7\pmod{8}.
$$
\end{corollary}

Using Theorems \ref{thm2} and \ref{thm4}, we conclude the following corollary.
\begin{corollary}\label{thm4-cor2}
Let $p$ be an odd prime such that $p\equiv 3\pmod{8}$. Then for any $\epsilon$ with $0<\epsilon < \frac{1}{2}$, we have
$$
Q(p, S_2\backslash S_4)- \frac{1}{2}\left\lfloor \frac{p-1}{4} \right\rfloor  \gg_\epsilon p^{\frac{1}{2}-\epsilon}.
$$
\end{corollary}

\section{Preliminaries}

In this section, we shall state many useful results as follows.

\begin{theorem}\label{polya1} (Polya-Vinogradov) 
Let $p$ be any odd prime and $\chi$ be a non-principal Dirichlet character modulo $p$. Then, for any integers $0\leq M < N\leq p-1$, we have 
$$
\left|\sum_{m=M}^N \chi(m)\right| \leq \sqrt{p} \log p.$$
\end{theorem}

Let us define the following counting functions as follows.  Let 
\begin{equation}\label{resi-count1}
f(x)=\displaystyle\frac{1}2\left(1+\left(\frac{x}{p}\right)\right)  \mbox{ for all  } x\in(\mathbb{Z}/p\mathbb{Z})^*
\end{equation}
and
\begin{equation}\label{resi-count2}
g(x)=
\displaystyle\frac12 \left(1-\left(\frac{x}{p}\right)\right)  \mbox{ for all } x\in (\mathbb{Z}/p\mathbb{Z})^*
\end{equation}
 where
$\displaystyle{\left(\frac{\cdot}{p}\right)}$ is the Legendre symbol. Then, we have
$$f(x)=\left\{\begin{array}{ll}
              1;&\mbox{if $x$ is a quadratic residue} \pmod p,\\
              0;&\mbox{otherwise}.
                   \end{array}\right.$$
and
$$g(x)=\left\{\begin{array}{ll}
              1;&\mbox{if $x$ is a quadratic nonresidue} \pmod p,\\
              0;&\mbox{otherwise}.
                  \end{array}\right.
$$

In the following lemma, we prove the ``expected'' result.

\bigskip
\begin{lemma}\label{expected}
For an integer $k\geq 1$ and an odd prime $p$, let $S_k = kI$ where $I$ is the interval $I =\{1, 2, \ldots, [(p-1)/k]\}.$  Then 
\begin{equation}\label{eqn 3}
Q(p, S_k)=\frac{1}{2}\left[\frac{p-1}{k}\right] +\frac12 \left(\frac{k}{p}\right)\sum_{m=1}^{(p-1)/k}\left(\frac{m}{p}\right)
\end{equation}
and hence 
$$
Q(p, S_k) = \frac{1}{2}\left[\frac{p-1}{k}\right] +O(\sqrt{p}\log p).
$$
The same expressions hold for $N(p, S_k)$ as well.
\end{lemma}

\begin{proof}
We prove for $Q(p, S_k)$ and the proof of $N(p, S_k)$ follows analogously. Let $\psi_k$ be the characteristic function for $S_k$ which is defined as 
$$ 
\psi_k(m) = \left\{\begin{array}{ll}
1; &\mbox{ if } m\in S_k,\\
0; &\mbox{ if } m\not\in S_k.
\end{array}\right.
$$
Now, by \eqref{resi-count1}, we see that 
\begin{eqnarray}\label{countingfunction}
Q(p,S_k) &=& \sum_{m\in S_k}f(m) = \sum_{m=1}^{p-1} \psi_k(m) f(m) = \frac{1}{2}\sum_{m=1}^{p-1} \psi_k(m) \left(1+\left(\frac{m}{p}\right)\right) \nonumber\\
&=& \frac{1}{2}\left[\frac{p-1}{k}\right] +\frac12 \left(\frac{k}{p}\right)\sum_{m=1}^{(p-1)/k}\left(\frac{m}{p}\right),
\end{eqnarray}
which proves \eqref{eqn 3}. Then, by Theorem \ref{polya1}, we get 
$$
Q(p, S_k) = \frac{1}{2}\left[\frac{p-1}{k}\right] + O(\sqrt{p}\log p).
$$
This finishes the proof. 
\end{proof}

Let $q >1$ be a positive integer and let $\psi$ be a non-trivial quadratic character modulo $q$. Let $L(s, \psi) = \displaystyle\sum_{n=1}^{\infty}\frac{\psi (n)}{n^s}$ be the Dirichlet L-function associated to $\psi$.  Since $\psi$  is a non-trivial homomorphism,  $L(s, \psi)$ admits the following Euler product expansion 
$$
L(s, \psi) = \prod_{p\nmid q}\left(1-\frac{\psi(p)}{p^s}\right)
$$
for all complex number $s$ with $\Re(s) > 1$. This, in particular, shows that $L(s, \psi) >0$ for all real number $s >1$. By continuity, it follows that  $L(1, \psi) \geq 0$. Dirichlet proved that $L(1, \psi) \ne 0$ in order to prove the infinitude of prime numbers in an arithmetic progression. Hence, it follows that $L(1, \psi) > 0$ for all non-trivial quadratic character $\psi$. Since $L(1, \psi) > 0$, it is natural to expect some non-trivial lower bound as a function of  $q$.  
This is what was proved by Landau-Siegel in the following theorem. The proof  can be found in \cite{montu}.

\begin{theorem} \label{estimate-lower bound}
Let $q >1$ be a positive integer and $\psi$ be a non-trivial quadratic character modulo $q$. Then for each $\epsilon > 0$, there exists a constant $C(\epsilon) > 0$ such that $$L(1, \psi) > \frac{C(\epsilon)}{q^{\epsilon}}.$$
\end{theorem}

The following lemma is crucial for our discussions. This lemma connects the sum of Legendre symbols and the Dirichlet L-function associated to Legendre symbol via the famous Dirichlet class number formula for the quadratic field.  For an odd prime $p$, the Legendre symbol $\displaystyle\left(\frac{\cdot}p\right) = \chi_p(\cdot)$ is a quadratic Dirichlet  character modulo $p$. We also define a  character 
$$
\chi_4(n) =\left\{\begin{array}{ll}
(-1)^{(n-1)/2}; & \mbox{ if $n$ is odd},\\
0; &\mbox{ otherwise}.
\end{array}\right.
$$
Then one can define the Dirichlet character $\chi_{4p}$ as $\chi_{4p}(n) = \chi_4(n)\chi_p(n)$ for any odd prime $p$ and similarly, we can define $\chi_{3p}(n) = \chi_3(n) \chi_p(n)$ for any odd prime $p>3$. Clearly, $\chi_{4p}$ and $\chi_{3p}$ are non-trivial and real quadratic Dirichlet characters. 

\bigskip
 
\begin{lemma} \label{wright1} (See for instance, Page 151, Theorem 7.2 and 7.4 in \cite{wright}) 
Let $p>3$ be an odd prime and   for any real number $\ell\geq 1$, we define
\begin{equation}\label{eq2}
S(1, \ell) = \displaystyle\sum_{1\leq m< \ell}\chi_p(m).
\end{equation}
 Then we have the following equalities.
\begin{enumerate}
\item For a prime $p\equiv 3\pmod{4}$, we have  
$$
S(1,p/2) = \frac{\sqrt{p}}{\pi} \left(2 - \chi_p(2)\right) L(1, \chi_p),$$
where $L(1, \chi_p)$ is the Dirichlet $L$-function; Also, we have 
$$
S(1, p/3) = \frac{\sqrt{p}}{2\pi} (3-\chi_p(3))L(1, \chi_p).
$$
\item For a prime  $p\equiv 1\pmod{4}$, we have 
$$
S(1,p/3) = \frac{\sqrt{3p}}{2\pi} L(1, \chi_{3p});$$
Also, we have 
$$
S(1, p/4) = \frac{\sqrt{p}}{\pi} L(1, \chi_{4p}).$$
\end{enumerate}
\end{lemma}

Now, we need the following lemma, which deals with the vanishing sums of  Legendre symbols. 
 This was proved in \cite{bruce}. For more such relations one may refer to \cite{joh1}.

\begin{lemma}\cite{bruce}\label{bruce}
Let $p$ be an odd prime. Then the following equalities hold true.
\begin{enumerate}
\item If $p\equiv 1\pmod{4}$, then we have $\displaystyle \sum _{n=1} ^ {(p-1)/2} \left( \frac{n}{p} \right) =0$.
\item If $p \equiv 3\pmod{8}$, then we have $ \displaystyle \sum _{n=1} ^{ \lfloor p/4 \rfloor} \left( \frac{n}{p} \right)= 0$.
\item If $ p \equiv 7 \pmod{8}$, then we have $ \displaystyle \sum _{\lceil p/4\rceil}^{\lfloor p/2\rfloor} \left( \frac{n}{p}\right) =0$.
\end{enumerate}
\end{lemma}
\medskip

\section{Proof of Theorem \ref{thm2}}

Let $p$ be a given odd prime. We want to estimate the quantity $Q(p, S_2)$. Therefore, by \eqref{countingfunction}, we get 
\begin{equation}\label{1eq1}
Q(p, S_2)  = \frac{1}{2}\left[\frac{p-1}{2}\right] +\frac{1}{2} \left(\frac{2}{p}\right)\sum_{n=1}^{(p-1)/2}\left(\frac{n}{p}\right).
\end{equation}
Now, we consider three cases as follows.

\bigskip

\noindent{\bf Case 1.} $p\equiv 1 \pmod{4}$

\smallskip

In this case,  since  
$\displaystyle \sum_{n=1}^{(p-1)/2} \left(\frac{n}{p}\right) = 0$, by Lemma \ref{bruce}  (1),  the equation \eqref{1eq1}  reduces to 
 $$
  Q(p, S_2) = \frac{p-1}{4},$$
 which is as desired.

\bigskip

\noindent{\bf Case 2.} $p\equiv 3 \pmod{8}$

\smallskip

By Lemma \ref{wright1} (1) and by \eqref{1eq1}, we get
$$
Q(p, S_2) = \frac 12\left[\frac{p-1}{2}\right]  +\frac{\sqrt{p}}{\pi} \left(2 - \chi_p(2)\right) L(1, \chi_p).
$$
In this case, we know that $ \left(\frac{2}{p}\right) = -1$. Therefore, we get 
$$
Q(p, S_2) = \frac 12\left[\frac{p-1}{2}\right]  + 3\frac{\sqrt{p}}{\pi} L(1, \chi_p).
$$
Let $\epsilon$ be any real number such that $0< \epsilon < \frac{1}{2}$. Then by  Theorem \ref{estimate-lower bound}, we get 
$$
Q(p, S_2) - \frac 12\left[\frac{p-1}{2}\right]  \gg_\epsilon p^{\frac{1}{2}-\epsilon},
$$
as desired.

\bigskip

\noindent{\bf Case 3.} $p\equiv 7\pmod{8}$.

\smallskip

Since $p\equiv 7\pmod{8}$, we know that $\displaystyle\left(\frac{2}{p}\right) = 1$. Therefore, by Lemma \ref{wright1} (1) and by \eqref{1eq1}, we get
$$
Q(p, S_2) = \frac{1}{2}\left[\frac{p-1}{2}\right] +  \frac{\sqrt{p}}{\pi} L(1, \chi_p) = \frac{1}{2}\left[\frac{p-1}{2}\right] + \frac{\sqrt{p} L(1,\chi_p)}{\pi}.
$$
Let  $\epsilon$ be any real number such that $0<\epsilon < \frac{1}{2}$. Then by  Theorem \ref{estimate-lower bound} we get
$$
Q(p, S_2) - \frac{1}{2}\left[\frac{p-1}{2}\right] \gg_\epsilon p^{\frac{1}{2}-\epsilon}
$$
which proves the theorem. $\hfill\Box$

\section{Proof of Theorem \ref{thm3}}

Let $p$ be a given odd prime. We want to estimate the quantity $Q(p, S_3)$. Therefore, by \eqref{countingfunction}, we get, 
\begin{equation}\label{2eq1}
Q(p, S_3)  = \frac{1}{2}\left[\frac{p-1}{3}\right] + \left(\frac{3}{p}\right)\sum_{n=1}^{(p-1)/3}\left(\frac{n}{p}\right).
\end{equation}
Now, we consider the following cases.

\bigskip

\noindent {\bf Case 1. }  $ p \equiv 1 \pmod{12}$

\smallskip

Note that, in this case, we have $\displaystyle\left( \frac{3}{p} \right) =1$. By \eqref{2eq1} and by Lemma \ref{wright1} (2), we get
\begin{eqnarray*}
Q(p, S_3) - \frac{1}{2}\left(\frac{p-1}{3}\right) &=& \frac12 \frac{\sqrt{3p}}{2\pi} L(1, \chi_{3}\chi_p)\\
     & \geq & \frac{\sqrt{3p}}{4\pi}\frac{C(\epsilon)}{(3p)^{\epsilon}}\\
     &\gg_\epsilon& p^{\frac{1}{2}-\epsilon},
\end{eqnarray*} 
for any given $0< \epsilon < \frac{1}{2}$ in Theorem \ref{estimate-lower bound}.

\bigskip

\noindent {\bf Case 2.}  $ p \equiv 11 \pmod{12}$

\smallskip

In this case, we have,  $\displaystyle\left( \frac{3}{p} \right) =1$.
Then again by \eqref{2eq1} and by Lemma \ref{wright1} (1), we get
\begin{eqnarray*}
Q(p, S_3)& = &\frac{1}{2}\left[\frac{p-1}{3}\right] +\frac12 \frac{\sqrt{3p}}{2\pi} (3- \chi_p(3))L(1,\chi_p).
\end{eqnarray*} 
Hence 
$$
Q(p, S_3) - \frac{1}{2}\left[ \frac{p-1}{3}\right] \gg_\epsilon p^{\frac{1}{2}-\epsilon},
$$
for any  $0<\epsilon < \frac{1}{2}$ in Theorem \ref{estimate-lower bound}. $\hfill\Box$

\bigskip

\section{Proof of theorem \ref{thm4}}
At first, using the equation \eqref{countingfunction}, we note that 
\begin{equation}\label{eqn5}
 Q(p,S_4)=\frac{1}{2}\left[\frac{p-1}{4}\right] +\frac12 \left(\frac{4}{p}\right)\sum_{m=1}^{(p-1)/4}\left(\frac{m}{p}\right) = \frac{1}{2} \left[ \frac{p-1}{4} \right] + \frac 12 \sum_{m=1}^{(p-1)/4}\left(\frac{m}{p}\right).
\end{equation}

\bigskip

\noindent {\bf Case 1. }  $ p \equiv 1 \pmod{4}$

\smallskip

Now, we apply Lemma \ref{wright1} (2) in \eqref{eqn5} and we get
\begin{eqnarray*}
Q(p, S_4)& = &\frac{1}{2}\left(\frac{p-1}{4}\right) + \frac12 \frac{\sqrt{p}}{\pi}L(1, \chi_{4} \chi_{p}).
\end{eqnarray*} 
Hence 
$$
Q(p, S_4) - \frac{p-1}{8} \gg_\epsilon p^{\frac{1}{2}-\epsilon},
$$
for any  $0<\epsilon < \frac{1}{2}$ in Theorem \ref{estimate-lower bound}.

\bigskip

\noindent {\bf Case 2. }  $ p \equiv 3 \pmod{8}$

\smallskip

In this case, we apply Lemma \ref{bruce} (2) which says that $\displaystyle \sum_{n=1}^{[(p-1)/{4}]} \left( \frac{m}{p} \right) = 0 $. Hence, by \eqref{eqn5}, we get 
$$  
Q(p,S_4) =\frac{1}{2} \left[ \frac{p-1}{4} \right].
$$

\noindent {\bf Case 3. }  $ p \equiv 7 \pmod{8}$

\smallskip

First note that by Lemma \ref{bruce} (3), we have 
$$
\displaystyle \sum _{ \frac{p-1}{4} < m < \frac{p-1}{2}} \left(\frac{m}{p} \right) =0.
$$
Therefore, the equation \eqref{eqn5} can be rewritten as 
\begin{eqnarray*}
Q(p,S_4) &= &\frac{1}{2} \left[ \frac{p-1}{4} \right] + \frac 12 \sum_{1\leq  m \leq {(p-1)/4}} \left(\frac{m}{p}\right) + \frac 12 \sum _{(p-1)/{4} \leq m \leq  {(p-1)}/{2}}\left(\frac{m}{p} \right)\\
      & =& \frac{1}{2} \left[ \frac{p-1}{4} \right] + \frac{1}{2} \sum _{m=1} ^{ \frac{p-1}{2}} \left( \frac{m}{p} \right).
\end{eqnarray*} 
Now, by Lemma \ref{wright1} (1), we get
\begin{eqnarray*}
Q(p, S_4) &= &\frac 12 \left[ \frac{p-1}{4}\right] +\frac 12 \frac{\sqrt{p}}{\pi}L(1, \chi_p).
\end{eqnarray*} 
Hence 
$$
Q(p, S_4) - \frac{1}{2} \left[ \frac{p-1}{4}\right] \gg p^{\frac{1}{2} - \epsilon},
$$
for any   $0<\epsilon < \frac{1}{2}$ in Theorem \ref{estimate-lower bound}. This proves the result. $\hfill\Box$

\bigskip

\noindent{\bf Acknowledgement.} We thank Professor V. Kumar Murty for going through the manuscript  very carefully and for a suggestion to clear our doubts.

\end{document}